\documentclass{amsart}
\usepackage{amsmath}
\usepackage{amssymb}
\usepackage{amsfonts}

\newtheorem{theorem}{Theorem}
\theoremstyle{plain}

\newtheorem{corollary}{Corollary}

\newtheorem{definition}{Definition}

\newtheorem{lemma}{Lemma}

\newtheorem{remark}{Remark}

\numberwithin{equation}{section}
\input{tcilatex}

\begin{document}
\title[Ostrowski type inequalities]{Ostrowski type Inequalities for $m-$ and 
$\left( \alpha ,m\right) -$geometrically convex functions via
Riemann-Louville Fractional integrals}
\author{Mevl\"{u}t TUN\c{C}}
\address{Kilis 7 Aral\i k University, Faculty of Science and Arts,
Department of Mathematics, Kilis, 79000, Turkey.}
\email{mevluttunc@kilis.edu.tr}

\begin{abstract}
In this paper, some new inequalities of Ostrowski type established for the
class of $m-$ and $\left( \alpha ,m\right) -$geometrically convex functions
which are generalizations of geometric convex functions.
\end{abstract}

\maketitle

\section{Introduction}

The following result is known in the literature as Ostrowski's inequality 
\cite{ost}.

\begin{theorem}
\bigskip Let $f:\left[ a,b\right] \rightarrow 
\mathbb{R}
$ be a differentiable mapping on $\left( a,b\right) $ with the property that 
$\left\vert f^{\prime }\left( u\right) \right\vert \leq M$ for all $u\in
\left( a,b\right) .$ Then the following inequality holds:%
\begin{equation}
\left\vert f\left( x\right) -\frac{1}{b-a}\int_{a}^{b}f\left( u\right)
du\right\vert \leq M\left( b-a\right) \left[ \frac{1}{4}+\left( \frac{x-%
\frac{a+b}{2}}{b-a}\right) ^{2}\right]  \label{101}
\end{equation}%
for all $x\in \left[ a,b\right] .$ The constant $1/4$ is best possible in
the sense that it cannot be replaced by a smaller constant.
\end{theorem}

This inequality gives an upper bound for the approximation of the integral
average $\frac{1}{b-a}\int_{a}^{b}f\left( u\right) du$\ by the value $%
f\left( x\right) $ at point $x\in \left[ a,b\right] $. For recent results
and generalizations concerning Ostrowski's inequality, see \cite{alo}-\cite%
{dra} and the references therein.

The following notations is well known in the literature.

\begin{definition}
\bigskip A function $f:I\rightarrow 
\mathbb{R}
,$ $\emptyset \neq I\subseteq 
\mathbb{R}
,$ where $I$ is a convex set, is said to be convex on $I$ if inequality%
\begin{equation*}
f\left( tx+\left( 1-t\right) y\right) \leq tf\left( x\right) +\left(
1-t\right) f\left( y\right)
\end{equation*}%
holds for all $x,y\in I$ and $t\in \left[ 0,1\right] $.
\end{definition}

In particular in \cite{GH}, Toader introduced the class of $m-$convex
functions as a generalizations of convexity as the following:

\begin{definition}
The function $f:\left[ 0,b\right] \rightarrow 
\mathbb{R}
$\ is said to be $m-$convex, where $m\in \left[ 0,1\right] $, if for every $%
x,y\in \left[ 0,b\right] $\ and $t\in \left[ 0,1\right] $\ we have\qquad\ \
\ \ \ \ \ \ \ \ \ \ \ 
\begin{equation}
f\left( tx+m\left( 1-t\right) y\right) \leq tf\left( x\right) +m\left(
1-t\right) f\left( y\right)  \label{2}
\end{equation}%
\bigskip
\end{definition}

Moreover, in \cite{mih}, Mihe\c{s}an introduced the class of $\left( \alpha
,m\right) -$convex functions as the following:

\begin{definition}
The function $f:\left[ 0,b\right] \rightarrow 
\mathbb{R}
$\ is said to be $\left( \alpha ,m\right) -$convex, where $\left( \alpha
,m\right) \in \left[ 0,1\right] ^{2},$\ if for every $x,y\in \left[ 0,b%
\right] $\ and $t\in \left[ 0,1\right] $\ we have \ \ \ \ \ \ \ \ 
\begin{equation}
f\left( tx+m\left( 1-t\right) y\right) \leq t^{\alpha }f\left( x\right)
+m\left( 1-t^{\alpha }\right) f\left( y\right) .\   \label{3}
\end{equation}%
\bigskip
\end{definition}

In \cite{xi}, Xi \textit{et al}. introduced the class of $m-$ and $\left(
\alpha ,m\right) -$geometrically convex functions as the following:

\begin{definition}
\cite{xi}\label{d}Let $f\left( x\right) $ be a positive function on $[0,b]$
and $m\in (0,1].$ If%
\begin{equation}
f\left( x^{t}y^{m\left( 1-t\right) }\right) \leq \left[ f\left( x\right) %
\right] ^{t}\left[ f\left( y\right) \right] ^{m\left( 1-t\right) }
\label{d1}
\end{equation}%
holds for all $x,y\in \left[ 0,b\right] $\ and $t\in \left[ 0,1\right] ,$
then we say that the function $f\left( x\right) $ is $m-$geometrically
convex on $[0,b]$.
\end{definition}

Obviously, if we set $m=1$ in Definition \ref{d}, then $f$ is just the
ordinary geometrically convex on $\left[ 0,b\right] $.

\begin{definition}
\cite{xi}\label{dd}Let $f\left( x\right) $ be a positive function on $[0,b]$
and $\left( \alpha ,m\right) \in \left( 0,1\right] \times \left( 0,1\right]
. $ If 
\begin{equation}
f\left( x^{t}y^{m\left( 1-t\right) }\right) \leq \left[ f\left( x\right) %
\right] ^{t^{\alpha }}\left[ f\left( y\right) \right] ^{m\left( 1-t^{\alpha
}\right) }  \label{d2}
\end{equation}%
holds for all $x,y\in \lbrack 0,b]$ and $t\in \lbrack 0,1]$, then we say
that the function $f\left( x\right) $ is $\left( \alpha ,m\right) -$%
geometrically convex on $[0,b]$.
\end{definition}

Clearly, when we choose $\alpha =1$ in Definition \ref{dd}, then $f$ becomes
the $m-$geometrically convex function on $\left[ 0,b\right] $. A very useful
inequality will be given as following:

\begin{lemma}
\cite{xi} For $x,y\in \left[ 0,\infty \right) $ and $m,t\in \left( 0,1\right]
,$ if $x<y$ and $y\geq 1,$ then%
\begin{equation*}
x^{t}y^{m\left( 1-t\right) }\leq tx+\left( 1-t\right) y.
\end{equation*}
\end{lemma}

We give some necessary definitions and mathematical preliminaries of
fractional calculus theory which are used throughout this paper.

\begin{definition}
Let $f\in L_{1}[a,b].$ The Riemann-Liouville integrals $J_{a^{+}}^{\mu }f$
and $J_{b^{-}}^{\mu }f$ of order $\mu >0$ with $a\geq 0$ are defined by%
\begin{equation*}
J_{a^{+}}^{\mu }f\left( x\right) =\frac{1}{\Gamma (\mu )}\underset{a}{%
\overset{x}{\int }}\left( x-t\right) ^{\mu -1}f(t)dt,\text{ \ }x>a
\end{equation*}%
and%
\begin{equation*}
J_{b^{-}}^{\mu }f\left( x\right) =\frac{1}{\Gamma (\mu )}\underset{x}{%
\overset{b}{\int }}\left( t-x\right) ^{\mu -1}f(t)dt,\text{ \ }x<b
\end{equation*}%
respectively where $\Gamma (\mu )=\underset{0}{\overset{\infty }{\int }}%
e^{-t}u^{\mu -1}du.$ Here is $J_{a^{+}}^{0}f(x)=J_{b^{-}}^{0}f(x)=f(x).$
\end{definition}

In the case of $\mu =1$, the fractional integral reduces to the classical
integral. Several researchers have interested on this topic and several
papers have been written connected with \ fractional integral inequalities
see \cite{anastas}, \cite{dahmani}, \cite{zdahm}, \cite{zdah}, \cite{dahtab}%
, \cite{a}, \cite{zeki2} and \cite{a2}.

The aim of this study is to establish some Ostrowski type inequalities for
the class of functions whose derivatives in absolute value are $m-$ and $%
\left( \alpha ,m\right) -$ geometrically convex functions via
Riemann-Liouville fractional integrals.

\section{Ostrowski type inequalities for $m-$ and $\left( \protect\alpha %
,m\right) -$geometrically convex functions}

In order to prove our main theorems, we need the following lemma that has
been obtained in \cite{a2}:

\begin{lemma}
\label{ls}Let $f:\left[ a,b\right] \rightarrow 
\mathbb{R}
$ be a differentiable mapping on $\left( a,b\right) $ with $a<b.$ If $%
f^{\prime }\in L\left[ a,b\right] ,$ then for all $x\in \left[ a,b\right] $
and $\mu >0$ we have:%
\begin{eqnarray*}
&&\frac{\left( x-a\right) ^{\mu }+\left( b-x\right) ^{\mu }}{b-a}f\left(
x\right) -\frac{\Gamma \left( \mu +1\right) }{b-a}\left[ J_{x^{-}}^{\mu
}f\left( a\right) +J_{x^{+}}^{\mu }f\left( b\right) \right] \\
&=&\frac{\left( x-a\right) ^{\mu +1}}{b-a}\int_{0}^{1}t^{\mu }f^{\prime
}\left( tx+\left( 1-t\right) a\right) dt+\frac{\left( b-x\right) ^{\mu +1}}{%
b-a}\int_{0}^{1}t^{\mu }f^{\prime }\left( tx+\left( 1-t\right) b\right) dt
\end{eqnarray*}%
where $\Gamma (\mu )=\underset{0}{\overset{\infty }{\int }}e^{-t}u^{\mu
-1}du.$
\end{lemma}

\begin{theorem}
\label{t}Let $I\supset \left[ 0,\infty \right) $ be an open interval and $%
f:I\rightarrow \left( 0,\infty \right) $ is differentiable. If $f^{\prime
}\in L\left[ a,b\right] $ and $\left\vert f^{\prime }\right\vert $ is
decreasing and $\left( \alpha ,m\right) -$geometrically convex on $\left[
\min \left\{ 1,a\right\} ,b\right] $ for $a\in \left[ 0,\infty \right) ,$ $%
b\geq 1,$ and $\left\vert f^{\prime }\left( x\right) \right\vert \leq M\leq
1,$ and $\left( \alpha ,m\right) \in \left( 0,1\right] \times \left( 0,1%
\right] ,$ then the following inequality for fractional integrals with $\mu
>0$ holds:%
\begin{eqnarray}
&&\left\vert \frac{\left( x-a\right) ^{\mu }+\left( b-x\right) ^{\mu }}{b-a}%
f\left( x\right) -\frac{\Gamma \left( \mu +1\right) }{b-a}\left[
J_{x^{-}}^{\mu }f\left( a\right) +J_{x^{+}}^{\mu }f\left( b\right) \right]
\right\vert  \label{x0} \\
&\leq &\left[ \frac{\left( x-a\right) ^{\mu +1}+\left( b-x\right) ^{\mu +1}}{%
b-a}\right] \times K\left( \alpha ,m,\mu ;k\left( \alpha \right) \right) 
\notag
\end{eqnarray}%
where%
\begin{equation*}
k\left( \alpha \right) =\left\{ 
\begin{array}{cc}
M^{m}\int_{0}^{1}t^{\mu }M^{t\alpha \left( 1-m\right) }dt & ,M<1 \\ 
\frac{1}{\mu +1} & ,M=1%
\end{array}%
\right. \text{ .}
\end{equation*}
\end{theorem}

\begin{proof}
By Lemma \ref{ls} and since $\left\vert f^{\prime }\right\vert $ is
decreasing and $\left( \alpha ,m\right) $-geometrically convex on $\left[
\min \left\{ 1,a\right\} ,b\right] ,$ we have%
\begin{eqnarray*}
&&\left\vert \frac{\left( x-a\right) ^{\mu }+\left( b-x\right) ^{\mu }}{b-a}%
f\left( x\right) -\frac{\Gamma \left( \mu +1\right) }{b-a}\left[
J_{x^{-}}^{\mu }f\left( a\right) +J_{x^{+}}^{\mu }f\left( b\right) \right]
\right\vert \\
&\leq &\frac{\left( x-a\right) ^{\mu +1}}{b-a}\int_{0}^{1}t^{\mu }\left\vert
f^{\prime }\left( tx+\left( 1-t\right) a\right) \right\vert dt+\frac{\left(
b-x\right) ^{\mu +1}}{b-a}\int_{0}^{1}t^{\mu }\left\vert f^{\prime }\left(
tx+\left( 1-t\right) b\right) \right\vert dt \\
&\leq &\frac{\left( x-a\right) ^{\mu +1}}{b-a}\int_{0}^{1}t^{\mu }\left\vert
f^{\prime }\left( x^{t}a^{m\left( 1-t\right) }\right) \right\vert dt+\frac{%
\left( b-x\right) ^{\mu +1}}{b-a}\int_{0}^{1}t^{\mu }\left\vert f^{\prime
}\left( x^{t}b^{m\left( 1-t\right) }\right) \right\vert dt \\
&\leq &\frac{\left( x-a\right) ^{\mu +1}}{b-a}\int_{0}^{1}t^{\mu }\left\vert
f^{\prime }\left( x\right) \right\vert ^{t^{\alpha }}\left\vert f^{\prime
}\left( a\right) \right\vert ^{m\left( 1-t^{\alpha }\right) }dt+\frac{\left(
b-x\right) ^{\mu +1}}{b-a}\int_{0}^{1}t^{\mu }\left\vert f^{\prime }\left(
x\right) \right\vert ^{t^{\alpha }}\left\vert f^{\prime }\left( b\right)
\right\vert ^{m\left( 1-t^{\alpha }\right) }dt \\
&\leq &\frac{\left( x-a\right) ^{\mu +1}}{b-a}\int_{0}^{1}t^{\mu
}M^{m+t^{\alpha }\left( 1-m\right) }dt+\frac{\left( b-x\right) ^{\mu +1}}{b-a%
}\int_{0}^{1}t^{\mu }M^{m+t^{\alpha }\left( 1-m\right) }dt \\
&=&\frac{M^{m}}{b-a}\int_{0}^{1}t^{\mu }M^{t^{\alpha }\left( 1-m\right) }dt%
\left[ \left( x-a\right) ^{\mu +1}+\left( b-x\right) ^{\mu +1}\right] .
\end{eqnarray*}%
If $0<\lambda \leq 1\leq \partial ,$ $0<u,v\leq 1,$ then%
\begin{equation}
\lambda ^{u^{v}}\leq \lambda ^{uv}.  \label{xx}
\end{equation}%
When $M\leq 1,$ by (\ref{xx}), we get that%
\begin{eqnarray}
&&\left\vert \frac{\left( x-a\right) ^{\mu }+\left( b-x\right) ^{\mu }}{b-a}%
f\left( x\right) -\frac{\Gamma \left( \mu +1\right) }{b-a}\left[
J_{x^{-}}^{\mu }f\left( a\right) +J_{x^{+}}^{\mu }f\left( b\right) \right]
\right\vert  \label{x1} \\
&\leq &\frac{M^{m}}{b-a}\left[ \left( x-a\right) ^{\mu +1}+\left( b-x\right)
^{\mu +1}\right] \int_{0}^{1}t^{\mu }M^{t^{\alpha }\left( 1-m\right) }dt 
\notag \\
&\leq &\frac{M^{m}}{b-a}\left[ \left( x-a\right) ^{\mu +1}+\left( b-x\right)
^{\mu +1}\right] \int_{0}^{1}t^{\mu }M^{t\alpha \left( 1-m\right) }dt. 
\notag
\end{eqnarray}%
The proof is completed.
\end{proof}

\begin{corollary}
Let $I\supset \left[ 0,\infty \right) $ be an open interval and $%
f:I\rightarrow \left( 0,\infty \right) $ is differentiable. If $f^{\prime
}\in L\left[ a,b\right] $ and $\left\vert f^{\prime }\right\vert $ is
decreasing and $m$-geometrically convex on $\left[ \min \left\{ 1,a\right\}
,b\right] $ for $a\in \left[ 0,\infty \right) ,$ $b\geq 1,$ and $\left\vert
f^{\prime }\left( x\right) \right\vert \leq M\leq 1,$ and $m\in \left( 0,1%
\right] ,$ then the following inequality for fractional integrals with $\mu
>0$ holds:%
\begin{eqnarray*}
&&\left\vert \frac{\left( x-a\right) ^{\mu }+\left( b-x\right) ^{\mu }}{b-a}%
f\left( x\right) -\frac{\Gamma \left( \mu +1\right) }{b-a}\left[
J_{x^{-}}^{\mu }f\left( a\right) +J_{x^{+}}^{\mu }f\left( b\right) \right]
\right\vert \\
&\leq &\left[ \frac{\left( x-a\right) ^{\mu +1}+\left( b-x\right) ^{\mu +1}}{%
b-a}\right] \times K\left( 1,m,\mu ;k\left( 1\right) \right)
\end{eqnarray*}%
where%
\begin{equation*}
k\left( 1\right) =\left\{ 
\begin{array}{cc}
\frac{1}{\mu +1} & ,M=1 \\ 
M^{m}\int_{0}^{1}t^{\mu }M^{\left( 1-m\right) t}dt & ,M\neq 1%
\end{array}%
\right. \text{ .}
\end{equation*}
\end{corollary}

\begin{proof}
We take $\alpha =1$ in (\ref{x0}), we get the required result.
\end{proof}

The corresponding version for powers of the absolute value of the first
derivative is incorporated in the following result:

\begin{theorem}
Let $I\supset \left[ 0,\infty \right) $ be an open interval and $%
f:I\rightarrow \left( 0,\infty \right) $ is differentiable. If $f^{\prime
}\in L\left[ a,b\right] $ and $\left\vert f^{\prime }\right\vert ^{q}$ is
decreasing and $\left( \alpha ,m\right) $-geometrically convex on $\left[
\min \left\{ 1,a\right\} ,b\right] $ for $a\in \left[ 0,\infty \right) ,$ $%
b\geq 1,$ $p,q>1$ and $\left\vert f^{\prime }\left( x\right) \right\vert
\leq M<1,$ $x\in \left[ \min \left\{ 1,a\right\} ,b\right] $ and $\left(
\alpha ,m\right) \in \left( 0,1\right) \times \left( 0,1\right) ,$ then the
following inequality for fractional integrals with $\mu >0$ holds:%
\begin{eqnarray}
&&\left\vert \frac{\left( x-a\right) ^{\mu }+\left( b-x\right) ^{\mu }}{b-a}%
f\left( x\right) -\frac{\Gamma \left( \mu +1\right) }{b-a}\left[
J_{x^{-}}^{\mu }f\left( a\right) +J_{x^{+}}^{\mu }f\left( b\right) \right]
\right\vert  \label{xxy} \\
&\leq &M^{m}\left( \frac{1}{p\mu +1}\right) ^{\frac{1}{p}}\left( \frac{%
M^{q\alpha \left( 1-m\right) }-1}{q\alpha \left( 1-m\right) \ln M}\right) ^{%
\frac{1}{q}}\left[ \frac{\left( x-a\right) ^{\mu +1}+\left( b-x\right) ^{\mu
+1}}{b-a}\right]  \notag
\end{eqnarray}%
where $p^{-1}+q^{-1}=1$.
\end{theorem}

\begin{proof}
\bigskip \bigskip By Lemma \ref{ls} and since $\left\vert f^{\prime
}\right\vert ^{q}$ is decreasing, and using the famous H\"{o}lder
inequality, we have%
\begin{eqnarray*}
&&\left\vert \frac{\left( x-a\right) ^{\mu }+\left( b-x\right) ^{\mu }}{b-a}%
f\left( x\right) -\frac{\Gamma \left( \mu +1\right) }{b-a}\left[
J_{x^{-}}^{\mu }f\left( a\right) +J_{x^{+}}^{\mu }f\left( b\right) \right]
\right\vert \\
&\leq &\frac{\left( x-a\right) ^{\mu +1}}{b-a}\int_{0}^{1}t^{\mu }\left\vert
f^{\prime }\left( tx+\left( 1-t\right) a\right) \right\vert dt+\frac{\left(
b-x\right) ^{\mu +1}}{b-a}\int_{0}^{1}t^{\mu }\left\vert f^{\prime }\left(
tx+\left( 1-t\right) b\right) \right\vert dt \\
&\leq &\frac{\left( x-a\right) ^{\mu +1}}{b-a}\left( \int_{0}^{1}t^{p\mu
}dt\right) ^{\frac{1}{p}}\left( \int_{0}^{1}\left\vert f^{\prime }\left(
x^{t}a^{m\left( 1-t\right) }\right) \right\vert ^{q}dt\right) ^{\frac{1}{q}}
\\
&&+\frac{\left( b-x\right) ^{\mu +1}}{b-a}\left( \int_{0}^{1}t^{p\mu
}dt\right) ^{\frac{1}{p}}\left( \int_{0}^{1}\left\vert f^{\prime }\left(
x^{t}b^{m\left( 1-t\right) }\right) \right\vert ^{q}dt\right) ^{\frac{1}{q}}
\end{eqnarray*}%
Since $\left\vert f^{\prime }\right\vert ^{q}$ is $\left( \alpha ,m\right) -$%
geometrically convex on $\left[ \min \left\{ 1,a\right\} ,b\right] $ and $%
\left\vert f^{\prime }\left( x\right) \right\vert \leq M<1$, we obtain that 
\begin{eqnarray*}
\int_{0}^{1}\left\vert f^{\prime }\left( x^{t}a^{m\left( 1-t\right) }\right)
\right\vert ^{q}dt &\leq &\int_{0}^{1}\left\vert f^{\prime }\left( x\right)
\right\vert ^{qt^{\alpha }}\left\vert f^{\prime }\left( a\right) \right\vert
^{mq\left( 1-t^{\alpha }\right) }dt \\
&\leq &\int_{0}^{1}M^{qt^{\alpha }+mq\left( 1-t^{\alpha }\right) }dt\leq
M^{mq}\int_{0}^{1}M^{qt^{\alpha }\left( 1-m\right) }dt \\
&\leq &M^{mq}\int_{0}^{1}M^{qt\alpha \left( 1-m\right) }dt=M^{mq}\frac{%
M^{q\alpha \left( 1-m\right) }-1}{q\alpha \left( 1-m\right) \ln M}
\end{eqnarray*}

and%
\begin{eqnarray*}
\int_{0}^{1}\left\vert f^{\prime }\left( x^{t}b^{m\left( 1-t\right) }\right)
\right\vert ^{q}dt &\leq &\int_{0}^{1}\left\vert f^{\prime }\left( x\right)
\right\vert ^{qt^{\alpha }}\left\vert f^{\prime }\left( b\right) \right\vert
^{mq\left( 1-t^{\alpha }\right) }dt \\
&\leq &M^{mq}\frac{M^{q\alpha \left( 1-m\right) }-1}{q\alpha \left(
1-m\right) \ln M}
\end{eqnarray*}%
and by simple computation%
\begin{equation*}
\int_{0}^{1}t^{p\mu }dt=\frac{1}{p\mu +1}.
\end{equation*}%
Hence, we have%
\begin{eqnarray*}
&&\left\vert \frac{\left( x-a\right) ^{\mu }+\left( b-x\right) ^{\mu }}{b-a}%
f\left( x\right) -\frac{\Gamma \left( \mu +1\right) }{b-a}\left[
J_{x^{-}}^{\mu }f\left( a\right) +J_{x^{+}}^{\mu }f\left( b\right) \right]
\right\vert \\
&\leq &M^{m}\left( \frac{1}{p\mu +1}\right) ^{\frac{1}{p}}\left( \frac{%
M^{q\alpha \left( 1-m\right) }-1}{q\alpha \left( 1-m\right) \ln M}\right) ^{%
\frac{1}{q}}\left[ \frac{\left( x-a\right) ^{\mu +1}+\left( b-x\right) ^{\mu
+1}}{b-a}\right]
\end{eqnarray*}%
which completes the proof.
\end{proof}

\begin{corollary}
Let $I\supset \left[ 0,\infty \right) $ be an open interval and $%
f:I\rightarrow \left( 0,\infty \right) $ is differentiable. If $f^{\prime
}\in L\left[ a,b\right] $ and $\left\vert f^{\prime }\right\vert ^{q}$ is
decreasing and $m$-geometrically convex on $\left[ \min \left\{ 1,a\right\}
,b\right] $ for $a\in \left[ 0,\infty \right) ,$ $b\geq 1,$ $p,q>1$ and $%
\left\vert f^{\prime }\left( x\right) \right\vert \leq M<1,$ $x\in \left[
\min \left\{ 1,a\right\} ,b\right] $ and $m\in \left( 0,1\right) ,$ then the
following inequality for fractional integrals with $\mu >0$ holds:%
\begin{eqnarray*}
&&\left\vert \frac{\left( x-a\right) ^{\mu }+\left( b-x\right) ^{\mu }}{b-a}%
f\left( x\right) -\frac{\Gamma \left( \mu +1\right) }{b-a}\left[
J_{x^{-}}^{\mu }f\left( a\right) +J_{x^{+}}^{\mu }f\left( b\right) \right]
\right\vert \\
&\leq &M^{m}\left( \frac{1}{p\mu +1}\right) ^{\frac{1}{p}}\left( \frac{%
M^{q\left( 1-m\right) }-1}{q\left( 1-m\right) \ln M}\right) ^{\frac{1}{q}}%
\left[ \frac{\left( x-a\right) ^{\mu +1}+\left( b-x\right) ^{\mu +1}}{b-a}%
\right]
\end{eqnarray*}%
where $1/p+1/q=1.$
\end{corollary}

\begin{proof}
We take $\alpha =1$ in (\ref{xxy}), we get the required result.
\end{proof}

A different approach leads to the following result.

\begin{theorem}
\label{ttt}Let $I\supset \left[ 0,\infty \right) $ be an open interval and $%
f:I\rightarrow \left( 0,\infty \right) $ is differentiable. If $f^{\prime
}\in L\left[ a,b\right] $ and $\left\vert f^{\prime }\right\vert ^{q}$ is
decreasing and $\left( \alpha ,m\right) $-geometrically convex on $\left[
\min \left\{ 1,a\right\} ,b\right] $ for $a\in \left[ 0,\infty \right) ,$ $%
b\geq 1,$ $q\geq 1$ and $\left\vert f^{\prime }\left( x\right) \right\vert
\leq M<1,$ $x\in \left[ \min \left\{ 1,a\right\} ,b\right] $ and $\alpha \in
\left( 0,1\right] ,$ $m\in \left( 0,1\right) ,$ then the following
inequality for fractional integrals with $\mu >0$ holds:%
\begin{eqnarray}
&&\left\vert \frac{\left( x-a\right) ^{\mu }+\left( b-x\right) ^{\mu }}{b-a}%
f\left( x\right) -\frac{\Gamma \left( \mu +1\right) }{b-a}\left[
J_{x^{-}}^{\mu }f\left( a\right) +J_{x^{+}}^{\mu }f\left( b\right) \right]
\right\vert  \label{xxx} \\
&\leq &M^{m}\left( \frac{1}{\mu +1}\right) ^{1-\frac{1}{q}}\left(
\int_{0}^{1}t^{\mu }M^{qt\alpha \left( 1-m\right) }dt\right) ^{\frac{1}{q}}%
\left[ \frac{\left( x-a\right) ^{\mu +1}+\left( b-x\right) ^{\mu +1}}{b-a}%
\right]  \notag
\end{eqnarray}
\end{theorem}

\begin{proof}
By Lemma \ref{ls} and since $\left\vert f^{\prime }\right\vert ^{q}$ is
decreasing, and using the power mean inequality, we have%
\begin{eqnarray*}
&&\left\vert \frac{\left( x-a\right) ^{\mu }+\left( b-x\right) ^{\mu }}{b-a}%
f\left( x\right) -\frac{\Gamma \left( \mu +1\right) }{b-a}\left[
J_{x^{-}}^{\mu }f\left( a\right) +J_{x^{+}}^{\mu }f\left( b\right) \right]
\right\vert \\
&\leq &\frac{\left( x-a\right) ^{\mu +1}}{b-a}\int_{0}^{1}t^{\mu }\left\vert
f^{\prime }\left( tx+\left( 1-t\right) a\right) \right\vert dt+\frac{\left(
b-x\right) ^{\mu +1}}{b-a}\int_{0}^{1}t^{\mu }\left\vert f^{\prime }\left(
tx+\left( 1-t\right) b\right) \right\vert dt \\
&\leq &\frac{\left( x-a\right) ^{\mu +1}}{b-a}\left( \int_{0}^{1}t^{\mu
}dt\right) ^{1-\frac{1}{q}}\left( \int_{0}^{1}t^{\mu }\left\vert f^{\prime
}\left( x^{t}a^{m\left( 1-t\right) }\right) \right\vert ^{q}dt\right) ^{%
\frac{1}{q}} \\
&&+\frac{\left( b-x\right) ^{\mu +1}}{b-a}\left( \int_{0}^{1}t^{\mu
}dt\right) ^{1-\frac{1}{q}}\left( \int_{0}^{1}t^{\mu }\left\vert f^{\prime
}\left( x^{t}b^{m\left( 1-t\right) }\right) \right\vert ^{q}dt\right) ^{%
\frac{1}{q}}
\end{eqnarray*}%
Since $\left\vert f^{\prime }\right\vert ^{q}$ is $\left( \alpha ,m\right) $%
-geometrically convex and $\left\vert f^{\prime }\left( x\right) \right\vert
\leq M<1$ and by (\ref{xx}), we obtain 
\begin{eqnarray*}
\int_{0}^{1}t^{\mu }\left\vert f^{\prime }\left( x^{t}a^{m\left( 1-t\right)
}\right) \right\vert ^{q}dt &\leq &\int_{0}^{1}t^{\mu }\left\vert f^{\prime
}\left( x\right) \right\vert ^{qt^{\alpha }}\left\vert f^{\prime }\left(
a\right) \right\vert ^{mq\left( 1-t^{\alpha }\right) }dt \\
&\leq &\int_{0}^{1}t^{\mu }M^{qt^{\alpha }+mq\left( 1-t^{\alpha }\right)
}dt\leq M^{mq}\int_{0}^{1}t^{\mu }M^{qt^{\alpha }\left( 1-m\right) }dt \\
&\leq &M^{mq}\int_{0}^{1}t^{\mu }M^{qt\alpha \left( 1-m\right) }dt
\end{eqnarray*}%
and similarly%
\begin{equation*}
\int_{0}^{1}t^{\mu }\left\vert f^{\prime }\left( x^{t}b^{m\left( 1-t\right)
}\right) \right\vert ^{q}dt\leq M^{mq}\int_{0}^{1}t^{\mu }M^{qt\alpha \left(
1-m\right) }dt
\end{equation*}%
Hence, we have%
\begin{eqnarray*}
&&\left\vert \frac{\left( x-a\right) ^{\mu }+\left( b-x\right) ^{\mu }}{b-a}%
f\left( x\right) -\frac{\Gamma \left( \mu +1\right) }{b-a}\left[
J_{x^{-}}^{\mu }f\left( a\right) +J_{x^{+}}^{\mu }f\left( b\right) \right]
\right\vert \\
&\leq &M^{m}\left( \frac{1}{\mu +1}\right) ^{1-\frac{1}{q}}\left(
\int_{0}^{1}t^{\mu }M^{qt\alpha \left( 1-m\right) }dt\right) ^{\frac{1}{q}}%
\left[ \frac{\left( x-a\right) ^{\mu +1}+\left( b-x\right) ^{\mu +1}}{b-a}%
\right]
\end{eqnarray*}%
which comletes the proof.
\end{proof}

\begin{corollary}
Let $I\supset \left[ 0,\infty \right) $ be an open interval and $%
f:I\rightarrow \left( 0,\infty \right) $ is differentiable. If $f^{\prime
}\in L\left[ a,b\right] $ and $\left\vert f^{\prime }\right\vert ^{q}$ is
decreasing and $m$-geometrically convex on $\left[ \min \left\{ 1,a\right\}
,b\right] $ for $a\in \left[ 0,\infty \right) ,$ $b\geq 1,$ $q\geq 1$ and $%
\left\vert f^{\prime }\left( x\right) \right\vert \leq M<1,$ $x\in \left[
\min \left\{ 1,a\right\} ,b\right] $ and $m\in \left( 0,1\right) ,$ then the
following inequality for fractional integrals with $\mu >0$ holds:%
\begin{eqnarray*}
&&\left\vert \frac{\left( x-a\right) ^{\mu }+\left( b-x\right) ^{\mu }}{b-a}%
f\left( x\right) -\frac{\Gamma \left( \mu +1\right) }{b-a}\left[
J_{x^{-}}^{\mu }f\left( a\right) +J_{x^{+}}^{\mu }f\left( b\right) \right]
\right\vert \\
&\leq &M^{m}\left( \frac{1}{\mu +1}\right) ^{1-\frac{1}{q}}\left(
\int_{0}^{1}t^{\mu }M^{qt\left( 1-m\right) }dt\right) ^{\frac{1}{q}}\left[ 
\frac{\left( x-a\right) ^{\mu +1}+\left( b-x\right) ^{\mu +1}}{b-a}\right]
\end{eqnarray*}
\end{corollary}

\begin{proof}
We take $\alpha =1$ in (\ref{xxx}), we get the required result.
\end{proof}

\begin{corollary}
Let $I\supset \left[ 0,\infty \right) $ be an open interval and $%
f:I\rightarrow \left( 0,\infty \right) $ is differentiable. If $f^{\prime
}\in L\left[ a,b\right] $ and $\left\vert f^{\prime }\right\vert ^{q}$ is
decreasing and geometrically convex on $\left[ a,b\right] $ for $a\in \left[
0,\infty \right) ,$ $b\geq 1,$ $q\geq 1$ and $\left\vert f^{\prime }\left(
x\right) \right\vert \leq M<1,$ $x\in \left[ a,b\right] ,$ then the
following inequality for fractional integrals with $\mu >0$ holds:%
\begin{eqnarray}
&&\left\vert \frac{\left( x-a\right) ^{\mu }+\left( b-x\right) ^{\mu }}{b-a}%
f\left( x\right) -\frac{\Gamma \left( \mu +1\right) }{b-a}\left[
J_{x^{-}}^{\mu }f\left( a\right) +J_{x^{+}}^{\mu }f\left( b\right) \right]
\right\vert  \label{set} \\
&\leq &M\left( \frac{1}{\mu +1}\right) \left[ \frac{\left( x-a\right) ^{\mu
+1}+\left( b-x\right) ^{\mu +1}}{b-a}\right]  \notag
\end{eqnarray}%
then, the inequality in (\ref{set}) is special version of Corollary 3 of 
\cite{a2}.
\end{corollary}

\begin{proof}
If we take $\alpha =1$ and $m\rightarrow 1$ in (\ref{xxx}), we get the
required result.
\end{proof}

\begin{corollary}
In Theorem \ref{ttt}, if we choose $\mu =1,$ then (\ref{xxx}) reduces
inequality above%
\begin{eqnarray*}
&&\left\vert f\left( x\right) -\frac{1}{b-a}\int_{a}^{b}f\left( y\right)
dy\right\vert \\
&\leq &M^{m}2^{\frac{1}{q}}\left( \frac{M^{q\alpha \left( 1-m\right) }-1}{%
\ln M^{q\alpha \left( 1-m\right) }}\left( 1-\frac{1}{\ln M^{q\alpha \left(
1-m\right) }}\right) \right) ^{\frac{1}{q}}\left[ \frac{\left( x-a\right)
^{2}+\left( b-x\right) ^{2}}{2\left( b-a\right) }\right]
\end{eqnarray*}
\end{corollary}

\begin{corollary}
Let $f,g,a,b,\mu ,q$\textit{\ be as in Theorem \ref{ttt}, and }$u,v>0$ with $%
u+v=1.$Then%
\begin{eqnarray}
&&\left\vert \frac{\left( x-a\right) ^{\mu }+\left( b-x\right) ^{\mu }}{b-a}%
f\left( x\right) -\frac{\Gamma \left( \mu +1\right) }{b-a}\left[
J_{x^{-}}^{\mu }f\left( a\right) +J_{x^{+}}^{\mu }f\left( b\right) \right]
\right\vert  \label{mm} \\
&\leq &M^{m}\left[ \frac{\left( x-a\right) ^{\mu +1}+\left( b-x\right) ^{\mu
+1}}{b-a}\right]  \notag \\
&&\times \left( \frac{1}{\mu +1}\right) ^{1-\frac{1}{q}}\left( \frac{u^{2}}{%
\mu +u}+\frac{v^{2}\left( M^{^{\frac{q\alpha \left( 1-m\right) }{v}%
}}-1\right) }{q\alpha \left( 1-m\right) \ln M}\right) ^{\frac{1}{q}}  \notag
\end{eqnarray}
\end{corollary}

\begin{proof}
By Lemma \ref{ls} and since $\left\vert f^{\prime }\right\vert ^{q}$ is
decreasing, and using the power mean inequality, we have%
\begin{eqnarray*}
&&\left\vert \frac{\left( x-a\right) ^{\mu }+\left( b-x\right) ^{\mu }}{b-a}%
f\left( x\right) -\frac{\Gamma \left( \mu +1\right) }{b-a}\left[
J_{x^{-}}^{\mu }f\left( a\right) +J_{x^{+}}^{\mu }f\left( b\right) \right]
\right\vert \\
&\leq &\frac{\left( x-a\right) ^{\mu +1}}{b-a}\left( \int_{0}^{1}t^{\mu
}dt\right) ^{1-\frac{1}{q}}\left( \int_{0}^{1}t^{\mu }\left\vert f^{\prime
}\left( x^{t}a^{m\left( 1-t\right) }\right) \right\vert ^{q}dt\right) ^{%
\frac{1}{q}} \\
&&+\frac{\left( b-x\right) ^{\mu +1}}{b-a}\left( \int_{0}^{1}t^{\mu
}dt\right) ^{1-\frac{1}{q}}\left( \int_{0}^{1}t^{\mu }\left\vert f^{\prime
}\left( x^{t}b^{m\left( 1-t\right) }\right) \right\vert ^{q}dt\right) ^{%
\frac{1}{q}}
\end{eqnarray*}%
Since $\left\vert f^{\prime }\right\vert ^{q}$ is $\left( \alpha ,m\right) $%
-geometrically convex and $\left\vert f^{\prime }\left( x\right) \right\vert
\leq M<1$ and by (\ref{xx}), we obtain 
\begin{equation*}
\int_{0}^{1}t^{\mu }\left\vert f^{\prime }\left( x^{t}a^{m\left( 1-t\right)
}\right) \right\vert ^{q}dt\leq M^{mq}\int_{0}^{1}t^{\mu }M^{qt\alpha \left(
1-m\right) }dt
\end{equation*}%
\begin{equation*}
\int_{0}^{1}t^{\mu }\left\vert f^{\prime }\left( x^{t}b^{m\left( 1-t\right)
}\right) \right\vert ^{q}dt\leq M^{mq}\int_{0}^{1}t^{\mu }M^{qt\alpha \left(
1-m\right) }dt
\end{equation*}%
By using the well known inequality $cd\leq uc^{\frac{1}{u}}+vd^{\frac{1}{v}}$%
, we get that%
\begin{eqnarray*}
\int_{0}^{1}t^{\mu }M^{qt\alpha \left( 1-m\right) }dt &\leq
&\int_{0}^{1}\left( ut^{\frac{^{\mu }}{u}}+vM^{^{\frac{qt\alpha \left(
1-m\right) }{v}}}\right) dt \\
&=&\frac{u}{\frac{^{\mu }}{u}+1}+v\frac{M^{^{\frac{q\alpha \left( 1-m\right) 
}{v}}}-1}{\ln M^{^{\frac{q\alpha \left( 1-m\right) }{v}}}} \\
&=&\frac{u^{2}}{\mu +u}+\frac{v^{2}\left( M^{^{\frac{q\alpha \left(
1-m\right) }{v}}}-1\right) }{q\alpha \left( 1-m\right) \ln M}
\end{eqnarray*}%
Hence, we have%
\begin{eqnarray*}
&&\left\vert \frac{\left( x-a\right) ^{\mu }+\left( b-x\right) ^{\mu }}{b-a}%
f\left( x\right) -\frac{\Gamma \left( \mu +1\right) }{b-a}\left[
J_{x^{-}}^{\mu }f\left( a\right) +J_{x^{+}}^{\mu }f\left( b\right) \right]
\right\vert \\
&\leq &M^{m}\left( \frac{1}{\mu +1}\right) ^{1-\frac{1}{q}}\left( \frac{u^{2}%
}{\mu +u}+\frac{v^{2}\left( M^{^{\frac{q\alpha \left( 1-m\right) }{v}%
}}-1\right) }{q\alpha \left( 1-m\right) \ln M}\right) ^{\frac{1}{q}}\left[ 
\frac{\left( x-a\right) ^{\mu +1}+\left( b-x\right) ^{\mu +1}}{b-a}\right]
\end{eqnarray*}%
which comletes the proof.
\end{proof}

\begin{remark}
In \ref{mm}, if we choose $q=1,$ then (\ref{mm}) reduces inequality above%
\begin{eqnarray*}
&&\left\vert \frac{\left( x-a\right) ^{\mu }+\left( b-x\right) ^{\mu }}{b-a}%
f\left( x\right) -\frac{\Gamma \left( \mu +1\right) }{b-a}\left[
J_{x^{-}}^{\mu }f\left( a\right) +J_{x^{+}}^{\mu }f\left( b\right) \right]
\right\vert \\
&\leq &M^{m}\left[ \frac{\left( x-a\right) ^{\mu +1}+\left( b-x\right) ^{\mu
+1}}{b-a}\right] \left( \frac{u^{2}}{\mu +u}+\frac{v^{2}\left( M^{^{\frac{%
\alpha \left( 1-m\right) }{v}}}-1\right) }{\alpha \left( 1-m\right) \ln M}%
\right)
\end{eqnarray*}
\end{remark}

\end{document}